\documentclass[11pt]{amsart}
\usepackage{lmodern}
\usepackage{amsmath, amsthm, amssymb, amsfonts}
\usepackage[normalem]{ulem}
\usepackage{hyperref}

\usepackage{mathrsfs}

\usepackage{verbatim} 
\usepackage{longtable}

\usepackage{mathtools}

\usepackage{tikz}
\usetikzlibrary{decorations.pathmorphing}
\tikzset{snake it/.style={decorate, decoration=snake}}

\usepackage{caption}

\usepackage{tikz-cd}
\usetikzlibrary{arrows}

\theoremstyle{plain}
\newtheorem{thm}{Theorem}[section]

\newtheorem{lem}[thm]{Lemma}
\newtheorem{prop}[thm]{Proposition}
\newtheorem{conj}[thm]{Conjecture}

\theoremstyle{definition}
\newtheorem{defn}[thm]{Definition}
\newtheorem{example}[thm]{Example}

\theoremstyle{remark}
\newtheorem{rmk}[thm]{Remark}

\newcommand{\BC}{{\mathbb{C}}}

\newcommand{\BP}{{\mathbb{P}}}
\newcommand{\BQ}{{\mathbb{Q}}}

\newcommand{\BZ}{{\mathbb{Z}}}

\newcommand{\CA}{{\mathcal A}}

\newcommand{\CE}{{\mathcal E}}
\newcommand{\CF}{{\mathcal F}}
\newcommand{\CG}{{\mathcal G}}
\newcommand{\CH}{{\mathcal H}}

\newcommand{\CK}{{\mathcal K}}

\newcommand{\CM}{{\mathcal M}}

\newcommand{\CO}{{\mathcal O}}
\newcommand{\CP}{{\mathcal P}}

\newcommand{\CX}{{\mathcal X}}

\newcommand{\Fp}{{\mathfrak{p}}}

\newcommand{\ch}{{\mathrm{ch}}}

\DeclareFontFamily{OT1}{rsfs}{}
\DeclareFontShape{OT1}{rsfs}{n}{it}{<-> rsfs10}{}
\DeclareMathAlphabet{\curly}{OT1}{rsfs}{n}{it}

\usepackage{tikz}
\usepackage{lmodern}
\usetikzlibrary{decorations.pathmorphing}

\begin{document}
\title[Dualizable abelian fibrations]{Dualizable abelian fibrations}
\date{\today}

\author[D. Maulik]{Davesh Maulik}
\address{Massachusetts Institute of Technology}
\email{maulik@mit.edu}

\author[J. Shen]{Junliang Shen}
\address{Yale University}
\email{junliang.shen@yale.edu}

\author[Q. Yin]{Qizheng Yin}
\address{Peking University}
\email{qizheng@math.pku.edu.cn}

\begin{abstract}

In his proof of the fundamental lemma of the Langlands program, Ngô initiated the study of the decomposition theorem for abelian fibrations. When an abelian fibration admits a duality structure, the decomposition theorem and the perverse filtration on cohomology exhibit rich structures. The purpose of these notes is to describe a framework for dualizable abelian fibrations and to discuss some recent progress and applications.
 
\end{abstract}

\maketitle

\setcounter{tocdepth}{1} 

\tableofcontents
\setcounter{section}{-1}

\section{Introduction}

Throughout, we work over the complex numbers $\BC$, and all Chow groups and (relative) Chow motives are taken with $\BQ$-coefficients.

An abelian fibration is a proper surjective morphism $f: M \to B$ between nonsingular varieties such that $f$ has equidimensional fibers and a general fiber is an abelian variety. Elliptic surfaces provide first examples of abelian fibrations, whose geometry and topology have been studied intensively since the foundational work of Kodaira and N\'eron in the 1960s.

In higher dimension, a complete classification of singular fibers analogous to that for elliptic surfaces is no longer possible. Instead, to help understand the topology of the fibration, a natural object to study is the (derived) pushforward
\begin{equation}\label{DT}
Rf_*\BQ_M \in D^b_c(B)
\end{equation}
lying in the bounded derived category of constructible sheaves. By the decomposition theorem of Beilinson, Bernstein, Deligne, and Gabber \cite{BBD}, the object (\ref{DT}) splits into shifted \emph{semisimple perverse sheaves}, \emph{i.e.}, direct sums of simple perverse sheaves which are the building blocks of the abelian category $\mathrm{Perv}(B) \subset D^b_c(B)$ of perverse sheaves. Every simple perverse sheaf is constructed from an irreducible subvariety $Z \subset B$ together with an irreducible local system on a dense open subset of $Z$. The subvariety $Z$ that appeared as above is called a \emph{support} of the morphism $f: M\to B$. In general, it is challenging to determine all the supports and the local systems associated with the decomposition theorem for (\ref{DT}).


Over 15 years ago, Ng\^o \cite{Ngo} observed that the decomposition theorem becomes more tractable for a large class of abelian fibrations, which he called \emph{$\delta$-regular weak abelian fibrations}. Ng\^o showed for such fibrations that remarkably limited geometric information suffices to determine \emph{all} the supports and local systems appearing in the decomposition theorem for~(\ref{DT}). Then Ng\^o applied this tool to the Hitchin system associated with a reductive group, from which he deduced the fundamental lemma of the Langlands program. A crucial geometric ingredient in Ng\^o's argument is the use of an action of a commutative $B$-group scheme $P$ on the total space $M$:
\[
\begin{tikzcd}
P \arrow[dr] 
\arrow[r, phantom, "\;\curvearrowright\;"] 
& M \arrow[d, "f"] \\
& B.
\end{tikzcd}
\]
The existence of such a group action is the main assumption for a $\delta$-regular weak abelian fibration, and it plays a key role in determining the supports and the local systems. We refer to \cite[Section 7]{Ngo}, \cite{Ngo_survey}, \cite[Section 1]{Shen_survey}, and \cite[Section 1]{MS_chi} for more detailed discussions. Besides the proof of the fundamental lemma, Ng\^o's techniques for the decomposition theorem of abelian fibrations have been applied to solving a number of open questions and conjectures, see \emph{e.g.}~\cite{dCRS, MS_HT, MS_chi, MS_PW, MT, MY}.

Recently, it has been observed that if the abelian fibration $f: M \to B$ admits a \emph{dual fibration} $f^\vee: M^\vee \to B$ satisfying certain nice properties, the object (\ref{DT}) exhibits richer structures. We highlight three directions in which the dualizable abelian fibration structure proves particularly powerful.

\begin{enumerate}
    \item[(a)] {\bf Locating tautological classes in perverse filtrations:} it was used in \cite{MSY} to give a proof of the $P=W$ conjecture \cite{dCHM} for the Hitchin system, in \cite{MSY, PSSZ} to give a proof of the $P=C$ conjecture \cite{KPS, KLMP} in the free range, and in \cite{Ghosh} to locate the Chern classes in the perverse filtration as conjectured in \cite{BMSY1}.
    \item[(b)] {\bf Constructing motivic decompositions:} it was used to verify the motivic decomposition conjecture of Corti--Hanamura for various abelian fibrations, including compactified Jacobian fibrations for integral locally planar curves \cite{MSY}, and Hitchin systems~\cite{MSY2}.
    \item[(c)] {\bf Analyzing multiplicative structures:} it was used in \cite{MSY} to prove that the perverse filtrations for many abelian fibrations are multiplicative with respect to the cup-product, generalizing a result of Oblomkov--Yun for certain compactified Jacobians~\cite{OY}; furthermore, the multiplicativity was used to construct the intrinsic cohomology ring of the universal Jacobian --- a ring structure on the universal fine compactified Jacobian which is independent of the choice of a natural compactification~\cite{BMSY2}.
\end{enumerate}

Roughly, an abelian fibration $f: M\to B$ is called \emph{dualizable}, if there is a dual fibration $f^\vee: M^\vee \to B$ whose general fiber is the dual abelian variety of the corresponding fiber of the original abelian fibration,
\[
\begin{tikzcd}
M \arrow[dr, "f"'] & & M^\vee \arrow[dl, "f^\vee"] \\
& B &,
\end{tikzcd}
\]
and there is an object $\CP$ on the relative product $M^\vee \times_B M$ which behaves like the Poincar\'e line bundle associated with an abelian scheme and its dual. The key observation in \cite{MSY} is that the correspondences induced by $\CP$, which are called the \emph{Fourier transforms}, characterize both the decomposition theorem for (\ref{DT}) and the associated perverse filtration. In particular, it suggests that even if we are only interested in the given abelian fibration $f: M \to B$, it is advantageous to consider its dual. 

We note that the structure of a dualizable abelian fibration is expected to enhance Ng\^o's group scheme action in the following sense: the fibration $f: M \to B$ is a partial compactification of (a torsor under) the group $P \to B$; a dual fibration $f^\vee: M^\vee \to B$ should then be a partial compactification of (a torsor under) the dual group $P^\vee \to B$. Furthermore, the property of $\delta$-regularity, which is a key ingredient in Ng\^o's theory, is a main geometric ingredient for the desired properties of the object $\CP$ for dualizable abelian fibrations \cite{MSY}; see also the proof of Theorem \ref{thm2.6} below. We refer to \cite{AF, Kim0, Kim} for relevant discussions regarding the group scheme $P$ and dual fibrations.

In Section \ref{sec1}, we review the work of Beauville \cite{B} and Deninger–Murre \cite{DM} for abelian schemes. This provides the main motivation for, and serves as a guiding example of, dualizable abelian fibrations. In Sections \ref{sec2} and \ref{sec3}, we discuss the definition, main results, and applications of dualizable abelian fibrations. The theory of dualizable abelian fibrations is presented largely following \cite{MSY} with minor modifications. Most results discussed in this paper are drawn from \cite{MSY, MSY2, Ghosh, BMSY1, BMSY2}.

\subsection{Acknowledgements}
This article closely follows the lectures given by the authors in the 2025 Summer Research Institute of Algebraic Geometry in Fort Collins.

D.M.~was supported by a Simons Investigator Grant.
J.S.~was supported by the NSF grant DMS-2301474 and a Sloan Research Fellowship.

\section{Abelian schemes}\label{sec1}

\subsection{Overview}

Let $f: A \to B$ be an abelian scheme of relative dimension $g$ over a nonsingular base~$B$. It admits a natural dual 
\begin{equation}\label{dual}
\begin{tikzcd}
A \arrow[dr, "f"'] & & A^\vee \arrow[dl, "f^\vee"] \\
& B
\end{tikzcd}
\end{equation}
with $A^\vee$ given by the identity component of the relative Picard $A^\vee:= \mathrm{Pic}^0(A)$. The dual $A^\vee$ is a moduli space relative over $B$ with a universal family given by the normalized Poincar\'e line bundle $\CP \in \mathrm{Pic}(A^\vee \times_B A)$. The main purpose of this section is to explain that the dual~(\ref{dual}) and the normalized Poincar\'e line bundle $\CP$ yield the following structural results:
\begin{enumerate}
    \item[(a)] The Fourier transform induced by $\CP$ yields a motivic decomposition for $Rf_*\BQ_A$; this decomposition is nowadays known as the Beauville decomposition.
    \item[(b)] The algebra objects $(Rf_*\BQ_A[-], \cup)$ and $(Rf^\vee_* \BQ_{A^\vee}[-], \ast)$ induced by the cup-product on~$A$ and the convolution product on $A^\vee$ are isomorphic via the Fourier transform of~(a).
    \item[(c)] The natural tautological classes are homogeneous with respect to the Beauville decomposition; this is a toy model of the $P=C$ phenomenon.
\end{enumerate}
All the results in this section are due to Beauville \cite{B} and Deninger--Murre \cite{DM}.

\subsection{Fourier transforms}\label{sec1.2}
The normalized Poincar\'e line bundle $\CP \in \mathrm{Pic}(A^\vee \times_B A)$ induces a Fourier--Mukai transform
\[
\mathrm{FM}_\CP: D^b\mathrm{Coh}(A^\vee) \xrightarrow{\simeq} D^b\mathrm{Coh}(A)
\]
whose inverse is given by the Fourier--Mukai kernel
\[
\CP^{-1}:= \CP^\vee \otimes p_A^*\omega_{A/B}[g] \in D^b\mathrm{Coh}(A \times_B A).
\]
Here $p_A: A\times_B A^\vee \to A^\vee$ is the natural projection and $\omega_{A/B}$ is the relative canonical bundle. 

We set
\begin{equation*} 
\mathfrak{F} = \sum_i \mathfrak{F}_i := \ch(\CP) \in \mathrm{CH}^*(A^\vee \times_B A), \quad \mathfrak{F}_i := \ch_i(\CP) \in \mathrm{CH}^i(A^\vee \times_B A).
\end{equation*}
We view $\mathfrak{F}$ as a mixed-degree Chow correspondence between $A^\vee$ and $A$, whose inverse is defined via Riemann--Roch:
\[
\mathfrak{F}^{-1} = \sum_i \mathfrak{F}^{-1}_i\in \mathrm{CH}^*(A \times_B A^\vee), \quad \mathfrak{F}^{-1}_i \in \mathrm{CH}^i(A \times_B A^\vee).
\]
The correspondences $\mathfrak{F}, \mathfrak{F}^{-1}$ are known as the \emph{Fourier transforms}.

We have a tautological expression of the relative diagonal classes
\begin{equation}\label{diag}
[\Delta_{A/B}] = \mathfrak{F} \circ \mathfrak{F}^{-1} \in \mathrm{CH}^g(A \times_B A), \quad [\Delta_{A^\vee/B}] = \mathfrak{F}^{-1} \circ \mathfrak{F} \in \mathrm{CH}^g(A^\vee\times_B A^\vee)
\end{equation}
where $\circ$ stands for the composition of correspondences.

The following \emph{Fourier vanishing} result plays a key role in the theory of Fourier transforms for abelian schemes.

\begin{prop}[Fourier vanishing] \label{prop1.1}
For $i+j \neq 2g$, we have
\[
\mathfrak{F}_i\circ \mathfrak{F}^{-1}_{j}=0 \in \mathrm{CH}^{i+j-g}(A \times_B A), \quad \mathfrak{F}^{-1}_i \circ \mathfrak{F}_j = 0 \in \mathrm{CH}^{i+j-g}(A^\vee\times_B A^\vee).
\]
\end{prop}

\begin{proof}
Considering the degree $k\neq 2g$ part of the second identity of (\ref{diag}) gives
\begin{equation}\label{prop1.1_1}
0 = \sum_{i+j=k} \mathfrak{F}^{-1}_i\circ \mathfrak{F}_{j}.
\end{equation}
Now for an integer $N$, we consider the ``multiplication by $N$'' map $[N]: A^\vee \to A^\vee$. The normalized Poincar\'e line bundle interacts with this map as
\[
\left([N] \times \mathrm{id}\right)^* \CP \simeq \CP^{\otimes N} \in \mathrm{Pic}\left( A^\vee \times A \right).
\]
In particular, the map $[N]\times \mathrm{id}: A^\vee\times A^\vee \to A^\vee\times A^\vee$ satisfies
\[
([N] \times \mathrm{id})^* \left( \mathfrak{F}^{-1}_i\circ \mathfrak{F}_j\right) = N^j\mathfrak{F}^{-1}_i\circ \mathfrak{F}_j
\]
Consequently, we apply the pullback along $([N] \times \mathrm{id})^*$ to (\ref{prop1.1_1}) for all $N$, yielding the second identity of Proposition \ref{prop1.1}. The first identity follows by a parallel argument.
\end{proof}

\subsection{Applications}
We discuss some applications of the duality (\ref{dual}) and the Fourier transforms $\mathfrak{F}, \mathfrak{F}^{-1}$. 

The decomposition theorem for the smooth morphism $f: A \to B$ reads
\begin{equation}\label{thm1.3_1}
Rf_* \BQ_A \simeq \bigoplus_{i=0}^{2g} R^if_* \BQ_A [-i]
\end{equation}
where each summand on the right-hand side is a (semisimple) local system.

The first application of the Fourier transform package is the construction of the Beauville decomposition --- a canonical motivic lifting of (\ref{thm1.3_1}) where the projectors are explicitly given by $\mathfrak{F}, \mathfrak{F}^{-1}$.

\begin{thm}[Motivic decomposition]\label{thm1.2}
    There is a decomposition of the relative Chow motive
    \[
    h(A) = \bigoplus_{i=0}^{2g} h_i(A)
    \]
    compatible with the cup-product. The projector $\mathfrak{q}_i$ for each summand $h_i(A)=(A, \mathfrak{q}_i, 0)$ is given~by 
    \[
    \mathfrak{q}_i:= \mathfrak{F}_i \circ \mathfrak{F}^{-1}_{2g-i} \in \mathrm{CH}^g(A \times_B A).
    \]
   Moreover, the sheaf-theoretic specialization yields a canonical multiplicative decomposition 
    \begin{equation}\label{can_decomp}
    Rf_*\BQ_A = \bigoplus_{i=0}^{2g} R^if_*\BQ_A[-i],
    \end{equation}
    \emph{i.e.}, the isomorphism \eqref{can_decomp} is compatible with the cup-product.
\end{thm}

The motivic decomposition follows directly from the first identity of (\ref{diag}), where the Fourier vanishing of Proposition \ref{prop1.1} is used to prove that the correspondences $\mathfrak{q}_i$ indeed form orthogonal projectors. The multiplicativity is deduced by showing the compatibility of the motivic decomposition and the ``multiplication by $N$'' map, which was already implicitly used in the proof of Proposition \ref{prop1.1}. In particular, consider the ``multiplication by $N$'' correspondence
\[
[N]^*: Rf_* \BQ_A \to Rf_*\BQ_A;
\]
the decomposition (\ref{can_decomp}) is recovered by the decomposition into eigenspaces of this correspondence with each local system $R^if_*\BQ_A$ the eigenspace corresponding to $N^i$.

Similarly, we consider the decomposition into eigenspaces for the dual fibration
\begin{equation}\label{can_decomp2}
    Rf^\vee_*\BQ_{A^\vee} = \bigoplus_{i=0}^{2g} R^if^\vee_*\BQ_{A^\vee}[-i].
\end{equation}
The decompositions (\ref{can_decomp}) and (\ref{can_decomp2}) further yield multiplicative decompositions of the cohomology
\[
H^*(A,\BQ) = \bigoplus_{i=0}^{2g} H^*_{(i)}(A, \BQ), \quad H^*(A^\vee,\BQ) = \bigoplus_{i=0}^{2g} H^*_{(i)}(A^\vee, \BQ)
\]
which split the Leray filtrations for $f: A \to B$ and $f^\vee: A^\vee\to B$ respectively.

\begin{thm}[Fourier stability]\label{thm1.3_2}
    The Fourier transforms
\begin{equation}\label{Fourier_stab}
\begin{tikzcd}
Rf_*\BQ_A[-] \arrow[r, bend left=30, "{\mathfrak{F}^{-1}}"] 
  & Rf^\vee_* \BQ_{A^\vee}[-] \arrow[l, bend left=30, "{\mathfrak{F}}"]
\end{tikzcd}
\end{equation}
respect the decompositions \eqref{can_decomp} and \eqref{can_decomp2}. Moreover, the only nontrivial correspondences restricting to the direct sum components are
\[
\begin{tikzcd}
R^if_*\BQ_A[-i] \arrow[r, bend left=30, "{\mathfrak{F}_{2g-i}^{-1}}"] 
  & R^{2g-i}f^\vee_* \BQ_{A^\vee}[-i] \arrow[l, bend left=30, "{\mathfrak{F}_i}"].
\end{tikzcd}
\]
In particular, the $i$-th piece $H^*_{(i)}(A, \BQ)$ is characterized by the Fourier transform:
\begin{equation}\label{P=C_0}
H_{(i)}^*(A, \BQ) = \mathrm{Im}\left( \mathfrak{F}_{i}: H^*(A^\vee, \BQ) \to H^*(A, \BQ) \right).
\end{equation}
\end{thm}

In fact, Fourier vanishing implies that the motivic decompositions for $f: A\to B$ and~$f^\vee: A^\vee \to B$ respectively as given in Theorem \ref{thm1.2} are compatible with the Fourier transforms~$\mathfrak{F}, \mathfrak{F}^{-1}$. Sheaf-theoretic specializations yield Theorem \ref{thm1.3_2}.

Finally, we recall that the multiplication map 
\begin{equation}\label{mult_0}
\mathrm{mult}: A^\vee \times_B A^\vee \to A^\vee
\end{equation}
with respect to the group structure. This yields a \emph{convolution product}
\[
\left( Rf^\vee_* \BQ_{A^\vee}[-], \ast \right),
\]
whose restrictions to the summands are 
\[
* = \mathrm{mult}_*: R^{2g-i}f^\vee_*\BQ_{A^\vee}[-i] \otimes R^{2g-j}f^\vee_* \BQ_{A^\vee}[-j] \to  R^{2g-i-j}f^\vee_*\BQ_{A^\vee}[-i-j].
\]
The next theorem further enhances the Fourier stability of Theorem \ref{thm1.3_2}.

\begin{thm}[Convolution]\label{thm1.3_3}
The Fourier transforms induce an isomorphism of algebra objects
    \[
\begin{tikzcd}
\left(Rf_*\BQ_A[-], \cup\right) \arrow[r, bend left=30, "{\mathfrak{F}^{-1}}"] 
  & \left(Rf^\vee_* \BQ_{A^\vee}[-], \ast\right) \arrow[l, bend left=30, "{\mathfrak{F}}"]
\end{tikzcd}
\]
which enhances \eqref{Fourier_stab}.
\end{thm}

We note that Theorem \ref{thm1.3_3} is a \emph{decategorification} of the following statement: the monoidal structure $(D^b\mathrm{Coh}(A), \otimes)$, under the Fourier--Mukai transforms, intertwines with the convolution structure $(D^b\mathrm{Coh}(A^\vee), \ast)$, where the convolution kernel is given by the structure sheaf of the graph of the multiplication map:
\[
\CO_{\mathrm{mult}} \in \mathrm{Coh}\left(A^\vee \times_B A^\vee \times_B A^\vee\right).
\]

\subsection{A toy model for $P=C$}\label{sec1.4}
The $P=C$ phenomenon is a match of two structures from different origins for certain moduli spaces of sheaves which admit abelian fibrations. More precisely, it asserts that the location of a tautological class in the perverse filtration associated with the abelian fibration is determined exactly by its Chern grading --- the latter being defined intrinsically without reference to the abelian fibration. The $P=W$ conjecture of de Cataldo--Hausel--Migliorini \cite{dCHM} for $\mathrm{GL}_n$ can be deduced from $P=C$ for the moduli space of stable Higgs bundles; see \cite[Section 0.3]{dCMS}. Indeed, all known proofs \cite{MS_PW, HMMS, MSY} of $P=W$ are obtained through proving $P=C$. More recently, the $P=C$ phenomenon has been discovered to hold for moduli spaces other than those in non-abelian Hodge theory \cite{KPS, KLMP}. In this section, we explain through a toy model that $P=C$ is a consequence of the Fourier transform package.

Let $C \to B$ be a family of nonsingular curves lying on a nonsingular surface $S$. For convenience, we assume that $S$ is projective and the family $C \to B$ admits a section \mbox{$\iota: B \hookrightarrow C$}.\footnote{These assumptions are only for the convenience of the discussion and can be removed; see \cite[Sections 4 and 5]{MSY}.}
The relative Jacobian 
\[
f: J_C \to B
\]
which parameterizes degree $0$ line bundles on the fibers is a self-dual abelian fibration
\begin{equation*}
\begin{tikzcd}
J_C \arrow[dr, "f"'] & & J_C = J_C^\vee \arrow[dl, "f^\vee =f"] \\
& B &
\end{tikzcd}
\end{equation*}
with $\CP$ the normalized Poincar\'e line bundle. The relative Jacobian $J_C$ can be viewed as moduli space of $1$-dimensional torsion sheaves on $S$ whose universal sheaf $\CF$ on $S\times J_C$ can be compared with $\CP$ via the Abel--Jacobi map $\mathrm{AJ}: C \to J_C$ induced by the section $\iota: B \hookrightarrow C$. More precisely, let $\mathrm{ev}: C \to S$ be the evaluation map; then Riemann--Roch implies that there is a normalization of the Chern character such that
\[
\widehat{\mathrm{ch}}(\CF) = (\mathrm{ev}\times \mathrm{id}_{J_C})_* (\mathrm{AJ}\times \mathrm{id}_{J_C})^*\mathrm{ch}(\CP) = (\mathrm{ev}\times \mathrm{id}_{J_C})_* (\mathrm{AJ}\times \mathrm{id}_{J_C})^*\mathfrak{F}.
\]
Define the tautological classes
\[
c_k(\gamma):= \int_{\gamma} \widehat{\mathrm{ch}}_k(\CF) \in H^*(J_C, \BQ), \quad \gamma \in H^*(S, \BQ).
\]

\begin{thm}[$P=C$]
    We have
    \[
    \prod_{i=1}^s c_{k_i}(\gamma_i) \in H^*_{(k_1+k_2+\cdots+k_s)}(J_C, \BQ).
    \]
    In particular, the location of each tautological class in the perverse filtration, which is equivalent to the Leray filtration for $f: J_C\to B$, is determined by the Chern grading.
\end{thm}

The proof is a direct consequence of Theorem \ref{thm1.3_2}. More precisely, the projection formula implies 
\[
c_k(\gamma) = \mathfrak{F}_k(\mathrm{ev}^*\mathrm{AJ}_*\gamma)
\]
which lies in $H^*_{(k)}(J_C, \BQ)$ by (\ref{P=C_0}); then the theorem follows from the multiplicativity of the decomposition $H^*(J_C, \BQ) = \bigoplus_{i} H^*_{(i)}(J_C, \BQ)$.

We see in this case that the Fourier transform explains the $P=C$ phenomenon. Furthermore, it clarifies why normalizations are required for the tautological classes \cite{dCHM, KPS, KLMP}: from the perspective of the Fourier transform, the normalizations arise from comparing a universal family with respect to the surface, to the normalized Poincar\'e line bundle with respect to the dual fibration.

\section{Dualizable abelian fibrations: definition and examples}\label{sec2}

\subsection{What is our purpose?}
Before introducing the (rather technical!) definition of dualizable abelian fibrations, we first discuss informally the underlying motivation.

In Section \ref{sec1}, we reviewed results for abelian schemes established by Beauville and Deninger--Murre over 30 years ago. A natural question is: which structures of abelian schemes can be \emph{generalized} to abelian fibrations with singular fibers? The examples of abelian fibrations we have in mind include:
\begin{enumerate}
    \item[(a)] \emph{Hitchin systems.} They are crucial in the $P=W$ conjecture, and serve as our main motivating example.
    \item[(b)] \emph{Compactified Jacobians.} The topology of compactified Jacobians has deep connections to knot invariants \cite{OY, ORS, Shen_survey} and the universal Abel--Jacobi theory \cite{BHPSS, BMP}.
    \item[(c)] \emph{Lagrangian fibrations.} They play a key role in the study of compact hyper-K\"ahler varieties.
\end{enumerate}

First, we should not expect that results parallel to Section \ref{sec1} hold for abelian fibrations when there are singular fibers. For example, the following theorem shows that the Beauville decomposition cannot exist even we include the ``mildest'' singular fibers.

\begin{thm}[\cite{BMSY1}]\label{thm2.1}
For $g\geq 3$, there exists a family $C \to B$ of curves of arithmetic genus $g$ with at worst one simple node, such that the total space of the relative compactified Jacobian~$\overline{J}_C$ is nonsingular, and the perverse filtration associated with $f: \overline{J}_C\to B$ does not admit any multiplicative splitting. Consequently, no isomorphism given by the decomposition theorem 
\[
Rf_*\BQ_{\overline{J}_C} \simeq \bigoplus_{i = 0}^{2g}\CH_{(i)}, \quad \CH_{(i)} := {^\Fp}H^{i + \dim B}(Rf_{*}\BQ_{\overline{J}_C})[-i - \dim B] \in D^b_c(B)
\]
is multiplicative.
\end{thm}

The failure of the Beauville decomposition may be explained by the fact that there is no appropriate extension of the ``multiplication by $N$'' structure over the singular fibers. Instead, we may hope to build our theory on a duality 
\[
\begin{tikzcd}
M \arrow[dr, "f"'] & & M^\vee \arrow[dl, "f^\vee"] \\
& B &
\end{tikzcd}
\]
with a pair of objects $\CP, \CP^{-1}$ on the relative product $M^\vee\times_B M$ that generalize the Poincar\'e line bundles. In many interesting examples, the existence of such a Poincar\'e sheaf was established by Arinkin \cite{A2}.

The Poincar\'e objects $\CP, \CP^{-1}$ still yield Fourier transforms
\[
\begin{tikzcd}
Rf_*\BQ_M[-] \arrow[r, bend left=30, "{\mathfrak{F}^{-1}}"] 
  & Rf^\vee_* \BQ_{M^\vee}[-] \arrow[l, bend left=30, "{\mathfrak{F}}"].
\end{tikzcd}
\]
However, the Fourier transform is not expected to satisfy Theorem \ref{thm1.3_3} with some convolution product defined on $M^\vee$. This was illustrated in the following example.

\begin{thm}[\cite{BMSY2}] \label{thm2.2}
For $g \geq 4$, there exist two birational abelian fibrations $f_1: M_1 \to B$ and $f_2: M_2 \to B$ given by fine compactified Jacobians associated with stable curves of genus $g$, which share a common dual $M^\vee \to B$, \emph{i.e.}, there exist Poincar\'e sheaves $\CP_i$ inducing derived equivalences
\[
\mathrm{FM}_{\CP_i}: \mathrm{}D^b\mathrm{Coh}(M^\vee) \xrightarrow{\simeq} D^b\mathrm{Coh}(M_i), \quad i=1,2,
\]
but whose cohomology rings are not isomorphic as $H^*(B. \BQ)$-algebras:
\[
\left(H^*(M_1, \BQ), \cup\right) \not\simeq  \left(H^*(M_2, \BQ), \cup \right).
\]
\end{thm}

In particular, Theorem \ref{thm2.2} shows that there is no well-defined convolution product on $M^\vee$ which is matched with the cup-product for both duals $M_1, M_2$ through a Fourier transform.

The two counterexamples above suggest that a theory of Fourier transforms for abelian fibrations (which should cover the abelian fibrations considered in Theorems \ref{thm2.1} and \ref{thm2.2}), if it exists, is considerably subtler than the corresponding theory for abelian schemes. We aim to single out a class of abelian fibrations that admit dual fibrations, for which the resulting duality enjoys favorable properties, while still including a sufficiently wide class of interesting geometric examples. We introduce this class of abelian fibration in the next section, which we call \emph{dualizable abelian fibrations}.

\subsection{Dualizable abelian fibrations}\label{sec2.2}

We begin by carefully defining what we mean by \emph{abelian fibrations}. We say that $f: M\to B$ is an abelian fibration of relative dimension~$g$, if
\begin{enumerate}
    \item[(a)] $M, B$ are nonsingular,
    \item[(b)] $f$ is proper and flat, and
    \item[(c)] $f: M\to B$ is generically an abelian scheme of relative dimension $g$.
\end{enumerate}
More precisely, the condition (c) says that there is a Zariski open subset $U \subset B$ over which $f_U: M_U \to U$ is an abelian scheme of relative dimension $g$. In applications, we sometimes consider the case where $f_U: M_U \to U $ is a torsor under an abelian scheme; the framework can be adapted accordingly to accommodate this situation; see Section \ref{sec3.4}.

In the following, we list the axioms for an abelian fibration of relative dimension $g$ to be \emph{dualizable}; clearly these axioms are modeled on the corresponding properties of abelian schemes that we discussed in Section \ref{sec1}.
\medskip

\noindent {\bf Axiom A: Dual abelian fibration exists.} There exists an abelian fibration
\[
f^\vee: M^\vee \to B
\]
such that there is a Zariski open subset $U \subset B$ over which we have dual abelian schemes: 
\begin{equation}\label{dual_U}
\begin{tikzcd}
M_U \arrow[dr, "f_U"'] & & M_U^\vee = \mathrm{Pic}^0(M_U) \arrow[dl, "f_U^\vee"] \\
& U &.
\end{tikzcd}
\end{equation}

\medskip

\noindent {\bf Axiom B: Full support.} Both $Rf_* \BQ_M$ and $Rf^\vee_* \BQ_{M^\vee}$ have full support. More precisely, every (shifted) simple perverse sheaf that appears in the decomposition theorem for either $Rf_* \BQ_M$ or $Rf^\vee_* \BQ_{M^\vee}$ has support $B$.

\begin{rmk}
Axiom B suggests that it may be better to restrict ourselves to $\delta$-regular weak abelian fibrations, for which Ng\^o's support techniques \cite{Ngo} apply. 
\end{rmk}
\medskip

\noindent {\bf Axiom C: Poincar\'e line bundle extends.} By Axiom A, we have the normalized Poincar\'e line bundle
\[
\CP_U \in \mathrm{Pic}\left( M^\vee_U \times_U M_U\right).
\]
We require that there is an extension 
\[
\CP \in D^b\mathrm{Coh}\left( M^\vee \times_B M \right)
\]
of the Poincar\'e line bundle, such that there exists $\CP^{-1} \in D^b\mathrm{Coh}(M \times_B M^\vee)$ with 
\[
\CP^{-1}\circ\CP\simeq \CO_{\Delta_{M^\vee/B}}, \quad \CP \circ \CP^{-1} \simeq \CO_{\Delta_{M/B}}.
\]
Here $\circ$ stands for the composition of the Fourier--Mukai kernels. 

In other words, Axiom C asserts that the Fourier--Mukai duality for (\ref{dual_U}) extends over $B$.
\medskip

\noindent {\bf Axiom C+: Poincar\'e line bundle extends nicely.} The objects 
\begin{equation}\label{multi_N}
\CP^{-1}\circ \CP^{\otimes N} \in D\mathrm{Coh}(M^\vee \times_B M^\vee), \quad \CP ^{\otimes N}\circ \CP^{-1} \in D\mathrm{Coh}(M \times_B M)
\end{equation}
are supported in codimension $g$ on $M^\vee \times_B M^\vee$ and $M\times_B M$ respectively for infinitely many integers $N\geq 1$.

\begin{rmk}
We note that the objects (\ref{multi_N}) above may \emph{not} be bounded, as $\CP, \CP^{-1}$ are not assumed to be \emph{perfect} on the possibly singular variety $M^\vee\times_B M$. Nevertheless, the codimension of their support is still well-defined. Over the abelian scheme locus $U$, the support of~(\ref{multi_N}) is given by the graphs of ``the multiplication by $N$'' maps for any $N\geq 1$, which clearly has codimension $g$. Axiom C+ says that there is no bigger component of the support over the singular locus of $f, f^\vee$.
\end{rmk}

\medskip

\noindent {\bf Axiom D: Convolution extends nicely.} Let $\Delta^{\mathrm{sm}}_{M/B}$ be the relative small diagonal in $M\times_B M \times_B M$. Then the object
\[
\CK: = \CP^{-1}\circ \CO_{\Delta^{\mathrm{sm}}_{M/B}}\circ \left( \CP \boxtimes \CP \right) \in D^b\mathrm{Coh}(M^\vee \times_B M^\vee \times_B M^\vee)
\]
is supported in codimension $g$, and the Zariski closure of
\[
\Gamma_{\mathrm{mult}} \subset M^\vee_U \times_U M^\vee \times_U M_U^\vee
\]
is the only component of pure codimension $g$ in its support.

\begin{rmk}
By definition, the Fourier--Mukai transform associated with $\CK$ defines a convolution product on $D^b\mathrm{Coh}(M^\vee)$ which intertwines the tensor product on $D^b\mathrm{Coh}(M)$; over the abelian scheme locus $U$, the object $\CK$ is exactly the structure sheaf of the graph $\Gamma_{\mathrm{mult}}$ of the multiplication map for the abelian scheme $M^\vee_U \to U$.
\end{rmk}

\begin{defn}[Dualizable abelian fibration]
    We say that $f: M \to B$ is a dualizable abelian fibration of relative dimen\-sion~$g$, if it is an abelian fibration of relative dimension $g$ satisfying Axioms A, B, C, C+,~D.
\end{defn}

\begin{rmk}
Although the group structures of $M_U, M^\vee_U$ cannot be extended over the singular fibers, Axioms C+ can be viewed as “poor man’s substitute” for the compatibility between the ``multiplication by $N$'' map and the Poincar\'e line bundle used in the proof of Proposition \ref{prop1.1} for abelian schemes, and Axiom D can be viewed as ``poor man's substitute" for the existence of a multiplication map (\ref{mult_0}).
\end{rmk}

\subsection{Examples}\label{sec2.3} As the axioms are modeled on abelian schemes, the following proposition is obvious.

\begin{prop}
    An abelian scheme $f: A \to B$ is a dualizable abelian fibration whose dual is given by the dual abelian scheme $f^\vee: A^\vee \to B$. 
\end{prop}

A main class of examples of dualizable abelian fibrations is given by compactified Jacobian fibrations. Let $C \to B$ be a proper flat family of integral locally planar curves of arithmetic genus $g$. We consider the relative compactified Jacobian parameterizing $(C_b, F_b)$ with $b\in B$ and $F_b$ a rank $1$ torsion-free sheaf on $C_b$ of degree $0$. Assume that $\overline{J}_C$ is nonsingular; then the natural morphism $f: \overline{J}_C \to B$ defines an abelian fibration.

\begin{thm}[\cite{MSY}]\label{thm2.6}
    With the notation as above, we have that $f: \overline{J}_C \to B$ is a dualizable abelian fibration whose dual is given by itself $f=f^\vee$.
\end{thm}

Theorem \ref{thm2.6} is \cite[Theorem 0.2]{MSY}; we discuss its proof briefly as follows.

\begin{proof}[Ingredients in the proof of Theorem \ref{thm2.6}]
The relative Jacobian 
\[
J_C: = \{(C_b, L_b)|~~ b\in B, ~~ L_b \in \mathrm{Pic}^0(C_b) \}
\]
is a commutative group scheme over $B$ which acts on $\overline{J}_C$ via tensor product
\begin{equation}\label{JC}
\begin{tikzcd}
J_C \arrow[dr] 
\arrow[r, phantom, "\;\curvearrowright\;"] 
& \overline{J}_C \arrow[d, "f"] \\
& B.
\end{tikzcd}
\end{equation}
For every point $b \in B$, we define $\delta(b)$ to be the dimension of the maximal affine subgroup of~$J_{C_b}$; this yields a constructible function
\[
\delta: B \to \BZ_{\geq 0}.
\]
For every irreducible $Z \subset B$, we define $\delta_Z$ to be the minimal value of $\delta(b)$ with $b\in Z$; equivalently, $\delta_Z$ calculates $\delta(b)$ for very general $b \in Z$. Following Ng\^o \cite{Ngo}, we say that the abelian fibration given by (\ref{JC}) is \emph{${\delta}$-regular} if we have for every irreducible $Z \subset B$:
\[
\delta_Z \leq \mathrm{codim}_B(Z).
\]
We recall the following lemma (\emph{c.f.}~\cite[Lemma 4.1]{MS_chi}).

\begin{lem}\label{lem2.7}
If $\overline{J}_C$ is nonsingular, then \eqref{JC} is $\delta$-regular.
\end{lem}

In fact, Theorem \ref{thm2.6} is proven by applying Lemma \ref{lem2.7} four(!) times:

Axiom A is clear by setting $f^\vee = f$.

Axiom B follows from the support theorem of Ng\^o, where the $\delta$-regularity is a key assumption in applying \cite{Ngo} (the first time).

The Poincar\'e sheaf $\CP$ for Axiom C was constructed by Arinkin \cite{A2} using Haiman's isospectral Hilbert schemes, and the $\delta$-regularity ensures that the induced Fourier--Mukai transform $\mathrm{FM}_\CP$ is a derived equivalence admitting an inverse $\mathrm{FM}_{\CP^{-1}}$ (the second time).

Finally, Axioms C+ and D are established by arguments similar to those in Arinkin’s proof of Axiom C \cite{A2} (see \cite[Sections 3.3 and 3.5]{MSY}); $\delta$-regularity is used in verifying both (the third and the forth times).
\end{proof}

We now discuss two interesting cases: the universal fine compactified Jacobian associated with a nondegenerate stability condition, and the parabolic Hitchin system associated with a general residue.  In the former case, we require the fact that Theorem \ref{thm2.6}, which concerns integral curves, can be further generalized to \emph{reduced} locally planar curves using \cite{MSY, MRV2}. 

\medskip

\noindent{\bf Universal fine compactified Jacobians.} By the discussion above, a universal fine compactified Jacobian over the Deligne--Mumford moduli space of stable curves $\overline{\CM}_{g,n}$ with respect to a nondegenerate stability condition in the sense of Kass--Pagani \cite{KP} is a dualizable abelian fibration.

More precisely, since a stable curve in $\overline{\CM}_{g,n}$ need not be irreducible, obtaining a separated, proper, and nonsingular universal compactified Jacobian requires the choice of a \emph{nondegenerate stability condition}. We use $\overline{J}^{\phi}_{g,n}$ to denote the universal fine compactified Jacobian, whose geometry relies heavily on the nondegenerate stability condition $\phi$.\footnote{In Theorem \ref{thm2.2}, all the spaces $M_1,M_2,M^\vee$ are given by universal fine compactified Jacobians.} 

The following theorem is obtained in \cite[Theorem 2.13]{MSY}.

\begin{thm}[\cite{BMSY2}] \label{thm2.8}
Assume $n\geq 1$. Let $\phi$ be a nondegenerate stability condition for $\overline{\CM}_{g,n}$, then $f^\phi: \overline{J}^\phi_{g,n} \to \overline{\CM}_{g,n}$ is a dualizable abelian fibration whose dual can be chosen as given by any nondegenerate stability condition $\phi'$:
\[
\begin{tikzcd}
\overline{J}^\phi_{g,n} \arrow[dr, "f^\phi"'] & & \overline{J}^{\phi'}_{g,n} \arrow[dl, "f^\vee = f^{\phi'}"] \\
& B &.
\end{tikzcd}
\]
\end{thm}

We note from the theorem that the dual fibration is not unique for a given abelian fibration; the flexibility of the choices of a dual fibration plays a crucial role in \cite{BMSY2}; see also Theorem \ref{thm3.7} below.

\medskip

\noindent{\bf Parabolic Hitchin systems.} A geometric object of particular interest to us is the Hitchin system associated with a curve $\Sigma$ (of genus $\geq 2$) and coprime rank $n$ and degree $d$. However, the Hitchin system is \emph{not} a dualizable abelian fibration: the full support property (Axiom B) fails \cite{dCHeM}, and it is unclear how to construct the Fourier–Mukai duality required by Axioms~C, C+, and D. In contrast, as we will discuss below, the parabolic Hitchin system associated with the same curve $\Sigma$, rank $n$, degree $d$, and a general residue is indeed a dualizable abelian fibration \cite{MSY2}. Some questions and conjectures for the usual Hitchin system can be reduced to the parabolic case \cite{HMMS, MSY2}.

Fix a point $p\in \Sigma$. Denote by ${M}^{\mathrm{par}}$ the moduli space of stable parabolic Higgs bundles~$(\CE, \theta, F_p)$, where $(\CE,\theta)$ is a meromorphic Higgs bundle
\[
(\CE, \theta), \quad \theta: \CE \to \CE\otimes \omega_\Sigma(p), \quad \mathrm{rk}(\CE) = n, \quad \mathrm{deg}(\CE) = d,
\]
and $F_p$ is a complete flag of $\CE$ over the point $p$ such that the residue $\mathrm{res}_p(\theta)$ of the Higgs field over the point $p$ preserves the flag $F_p$. Let $M^\eta_{n,d}$ be the moduli space of stable parabolic Higgs bundles of rank $n$ and degree $d$ with respect to a general residue $\eta$ at the point $p\in \Sigma$, with 
\[
f_d: M^\eta_{n,d} \to B
\]
the parabolic Hitchin system.

\begin{thm}[\emph{c.f.}~\cite{MSY2}]\label{thm2.9}
The parabolic Hitchin system $f_{d}: M^\eta_{n,d} \to B$ is a dualizable abelian fibration whose dual is given by the parabolic Hitchin system $f_{d'}: M^\eta_{n,d'} \to B$ with respect to the same residue $\eta$ and a degree $d'$ coprime to $n$,
\[
\begin{tikzcd}
M^\eta_{n,d} \arrow[dr, "f_d"'] & & M^\eta_{n,d'} \arrow[dl, "f^\vee = f_{d'}"] \\
& B &.
\end{tikzcd}
\]
\end{thm}

We note that Theorem \ref{thm2.9}, as stated here, is slightly inaccurate and requires modification. The issue is that the Hitchin system $f_d: M^\eta_{n,d} \to B$ does not have nonsingular fibers forming an abelian scheme; rather, they form a torsor under an abelian scheme. Nevertheless, these fibrations can still be treated within the framework of dualizable abelian fibrations with minor adjustments. We refer to \cite[Section 4]{MSY} and \cite[Section 2]{MSY2} for a more precise discussion.

\medskip

\noindent {\bf Further candidates?} We list two classes of potential candidates for dualizable abelian fibrations.

\begin{enumerate}
    \item[(a)] The work of Arinkin--Fedorov \cite{AF} constructed partial Fourier--Mukai transforms for Lagrangian fibrations with integral fibers; it is expected that the Arinkin--Fedorov construction eventually leads to a full Fourier--Mukai duality.
    \item[(b)]Recently, Bae--Molcho--Pixton \cite{BMP} studied a partial toroidal compactification of the universal principally polarized abelian varieties $\CX_g \to \CA_g$, and constructed a Fourier--Mukai (self-)duality; see \cite[Theorem 1.5]{BMP}. It is expected that this gives another example of a (self-)dualizable abelian fibration.
\end{enumerate}

\section{Dualizable abelian fibrations: results and applications}\label{sec3}

\subsection{Decomposition theorem}
Since our main goal is to understand the decomposition theorem \cite{BBD} for dualizable abelian fibrations, we begin with a brief discussion of the relevant general terminology.

Let $f: X \to Y$ be a proper and flat morphism between nonsingular varieties. Assume that each fiber is of dimension $g$. The decomposition theorem reads
\begin{equation}\label{DDTT}
Rf_*\BQ_{X} \simeq \bigoplus_{i = 0}^{2g}\CH_{(i)}, \quad \CH_{(i)} := {^\Fp}H^{i + \dim Y}(Rf_{*}\BQ_X)[-i - \dim Y] \in D^b_c(Y).
\end{equation}
The decomposition theorem is not canonical; what remain canonical are the perverse truncation functors ${^\Fp}\tau_{\leq \bullet}(-), {^\Fp}\tau_{\geq \bullet}(-)$ applied to the pushforward complex $Rf_{*}\BQ_X$. Setting
\[
P_iH^*(X, \BQ) := \mathrm{Im}\left(H^*(Y, {^\Fp}\tau_{\leq i + \dim Y}Rf_{*}\BQ_{X}) \to H^*(Y, Rf_{*}\BQ_{X}) = H^*(X, \BQ)\right),
\]
we obtain the perverse filtration
\[
P_0 H^*(X, \BQ) \subset P_1H^*(X, \BQ) \subset \cdots \subset P_{2g}H^*(X, \BQ) = H^*(X, \BQ)
\]
and the canonical identification
\[
\mathrm{Gr}_i^PH^*(X, \BQ) = H^*(B, \CH_{(i)}).
\]

The cup-product on $X$ yields the cup-product on the pushforward complex, 
\begin{equation}\label{cup}
\cup: Rf_* \BQ_X \otimes Rf_* \BQ_X \to Rf_* \BQ_X.
\end{equation}
We say that $f: X\to Y$ admits a \emph{multiplicative decomposition theorem} if there exists an isomorphism (\ref{DDTT}) which interacts with the cup-product (\ref{cup}) as
\[
\cup: \CH_{(i)} \otimes \CH_{(j)} \to \CH_{(i+j)}.
\]
By Theorem \ref{thm1.2}, an abelian scheme admits a multiplicative decomposition theorem; this was generalized to any smooth abelian fibration in \cite[Appendix A]{BMSY2}. In contrast, Theorem \ref{thm2.1} shows that compactified Jacobian fibrations $f: \overline{J}_C \to B$ associated with nodal curves may not admit a multiplicative decomposition theorem. 

In general, multiplicative decomposition theorems are rare. For example, most projective bundles $\BP(V) \to Y$ associated with a vector bundle $V\to Y$ do not admit a multiplicative decomposition theorem; this is because the Leray filtration (which is equivalent to perverse filtration in this case) does not admit a multiplicative splitting for most projective bundles~$\BP(V)$.

A weaker condition is the \emph{multiplicativity of the perverse truncation functors}, \emph{i.e.}, 
\begin{equation}\label{filtration_mult}
\cup: {^\Fp}\tau_{\leq i + \dim Y}(Rf_*\BQ_{X})  \otimes {^\Fp}\tau_{\leq j + \dim Y}(Rf_*\BQ_{X}) \to {^\Fp}\tau_{\leq i + j + \dim Y}(Rf_*\BQ_{X}), \quad i,j \in \BZ.
\end{equation}
Of course, if $f: X\to Y$ admits a multiplicative decomposition theorem, it has multiplicative perverse truncation functors; but the latter is much weaker. For example, any smooth and proper map $f: X\to Y$ has multiplicative perverse truncation functors. But a general proper flat map between nonsingular varieties may fail to have multiplicative perverse truncation functors; see \cite[Exercise 5.6.8]{Park} and the discussion in \cite{dC_Survey}.

If (\ref{filtration_mult}) holds for $f: X\to Y$, then the cup-product still yields a cup-product on the summands:
\[
\overline{\cup}: \CH_{(i)} \otimes \CH_{(j)} \to \CH_{(i+j)},
\]
which we call the \emph{induced cup-product}. In general, if $f: X\to Y$ does not admit a multiplicative decomposition theorem, the algebra object
\[
(\bigoplus_{i=0}^{2g} \CH_{(i)}, \overline{\cup})
\]
does not recover the algebra object
\[
(Rf_* \BQ_X, \cup).
\]
The first algebra object is a degeneration of the second algebra object.

Finally, we say that $f: X \to Y$ admits a \emph{motivic decomposition}, if there exists an isomorphism of relative Chow motives
\[
h(X) = \bigoplus_{i=0}^{2g} h_i(X), \quad h(X) = (X, [\Delta_{X/Y}], 0), \,\, h_i(X) = (X, \mathfrak{q}_i, 0)
\]
whose sheaf-theoretic specialization recovers an isomorphism (\ref{DDTT}). In other words, there exists an orthogonal decomposition of the relative diagonal class
\[
[\Delta_{X/Y}] = \sum_{i=0}^{2g} \mathfrak{q}_i \in \mathrm{CH}_{\dim X} (X\times_Y X)
\]
whose associated (self-)correspondences on $Rf_* \BQ_X$ recover a decomposition (\ref{DDTT}). The motivic decomposition conjecture of Corti--Hanamura \cite{CH} predicts that a motivic decomposition always exists.

\begin{conj}[Motivic decomposition conjecture \cite{CH}]\label{conj3.1}
Any proper $f: X\to Y$ between nonsingular varieties admits a motivic decomposition.
\end{conj}

Theorem \ref{thm1.2} shows that the conjecture holds for abelian schemes with the summands $h_i(X)$ constructed from the Fourier transforms.

\subsection{Main theorem}
We discuss geometric consequences for an abelian fibration to be dualizable in the sense of Section \ref{sec2.2}.

The following theorem combines results in \cite{MSY} and \cite{BMSY2}.

\begin{thm}[\cite{MSY, BMSY2}]\label{thm3.2} 
Let $f: M \to B$ be a dualizable abelian fibration of relative dimension~$g$. Then the following hold.
\begin{enumerate}
    \item[(a)] The morphism $f: M \to B$ admits a motivic decomposition, \emph{i.e.}, Conjecture \ref{conj3.1} holds for $f: M\to B$.
    \item[(b)] The morphism $f: M \to B$ has multiplicative perverse truncation functors.
    \item[(c)] Let $f^\vee: M^\vee \to B$ be a dual to $f: M \to B$. Then the class \[[\overline{\Gamma}_{\mathrm{mult}}] \in \mathrm{CH}_{\dim B +2g}(M^\vee \times_B M^\vee \times_B M^\vee)
    \]
    induces a convolution product 
    \[
\ast = [\overline{\Gamma}_{\mathrm{mult}}]_*: \CH^\vee_{(2g-i)}[2g - 2i] \otimes \CH^\vee_{(2g-j)}[2g - 2j] \to  \CH^\vee_{(2g-i-j)}[2g - 2i - 2j],
    \]
    and there is an isomorphism of algebra objects
    \[
    \left(\bigoplus_{i=0}^{2g} \CH_{(i)}, \overline{\cup}\right) \simeq  \left(\bigoplus_{i=0}^{2g} \CH^\vee_{(2g-i)}[2g - 2i], \ast \right).
    \]
    Here $\CH_{(i)}, \CH^\vee_{(i)}$ are defined as in \eqref{DDTT} for $f, f^\vee$ respectively.
\end{enumerate}
\end{thm}


Theorem \ref{thm3.2} partially generalizes structures of abelian schemes to dualizable abelian fibrations. The following are immediate consequences of Theorem \ref{thm3.2} in cohomology.

\begin{thm}\label{thm3.3}
   For a dualizable abelian fibration $f: M \to B$, we have the following.
   \begin{enumerate}
       \item[(a)] The perverse filtration $P_\bullet H^*(M, \BQ)$ is multiplicative, \emph{i.e.},
       \[
      \cup: P_iH^*(M, \BQ) \times P_{j}H^*(M, \BQ) \to P_{i+j}H^*(M, \BQ).
       \]
       In particular, the associated graded is a bigraded algebra over $H^*(B, \BQ)$ with the induced cup-product:
       \begin{equation}\label{intrinsic}
\left(       \bigoplus_{i,d} \mathrm{Gr}^P_i H^d(M, \BQ),  \overline{\cup} \right).
       \end{equation}
       \item[(b)] The bigraded $H^*(B, \BQ)$-algebra \eqref{intrinsic} is isomorphic to the bigraded $H^*(B, \BQ)$-algebra 
       \[
       \left(\bigoplus_{i,d} \mathrm{Gr}^P_{2g-i}H^{d+2g-2i}(M^\vee, \BQ), ~~~\ast \right)
       \]
       where the convolution product is induced by the class $[\overline{\Gamma}_{\mathrm{mult}}]$.
   \end{enumerate}
\end{thm}

Although the Fourier transforms do not appear \emph{directly} in the statement of Theorem \ref{thm3.2}, they play a crucial role in each part of its proof, in a manner analogous to the corresponding arguments for abelian schemes.

\begin{proof}[Ingredients of the proof of Theorem \ref{thm3.2}]
The Poincar\'e objects $\CP, \CP^{-1}$ given by Axiom C yield the Fourier transforms $\mathfrak{F}, \mathfrak{F}^{-1}$ as in the abelian scheme case --- they are naturally algebraic cycles of mixed degree on the relative product $M^\vee \times_B M$. In particular, we get the decompositions of the relative diagonal classes
\begin{equation}\label{diag2}
[\Delta_{M/B}] = \mathfrak{F} \circ \mathfrak{F}^{-1} \in \mathrm{CH}_{\dim M}(M \times_B M), \quad [\Delta_{M^\vee/B}] = \mathfrak{F}^{-1} \circ \mathfrak{F} \in \mathrm{CH}_{\dim M^\vee}(M^\vee\times_B M^\vee).
\end{equation}

Then Axiom C+ implies the following weaker version of the Fourier vanishing in Proposition~\ref{prop1.1}.

\begin{prop}[$\frac{1}{2}$ Fourier vanishing]\label{prop3.4}
   For $i+j < 2g$, we have
\[
\mathfrak{F}_i\circ \mathfrak{F}^{-1}_{j}=0 \in \mathrm{CH}_{\dim B + 3g -i-j}(M \times_B M), \quad \mathfrak{F}^{-1}_i \circ \mathfrak{F}_j = 0 \in \mathrm{CH}_{\dim B + 3g-i-j}(M^\vee \times_B M^\vee).
\]    
\end{prop}

Recall that the proof of Proposition \ref{prop1.1} relies on the ``multiplication by $N$'' map. The (weaker) replacement Axiom C+ recovers half of it (as in Proposition \ref{prop3.4}) using Adams operations; the detailed argument was given in \cite[Section 3.5]{MSY}.\footnote{We note that although the proof given in \cite[Section 3.5]{MSY} only handles compactified Jacobians, the essential geometric input is indeed Axiom C+.} 

It was then explained in \cite[Section 2.5.5]{MSY} that Proposition \ref{prop3.4} combined with the identities~(\ref{diag2}) produces motivic decompositions for both $f: M \to B$ and $f^\vee: M \to B$. Finally, Axiom A on the support was used to verify that the sheaf-theoretic specialization of the motivic decompositions we constructed indeed recovers the decomposition theorem for $f, f^\vee$; see~\cite[Section 2.5.2]{MSY}. This completes the proof of part (a).

As a byproduct of the proof of (a), we showed that the projector given by
\begin{equation}\label{P_cycle}
\mathfrak{p}_k: = \sum_{i\leq k}\mathfrak{F}_i\circ \mathfrak{F}^{-1}_{2g-i}: Rf_*\BQ_M \to Rf_* \BQ_M
\end{equation}
realizes exactly the perverse truncation functor
\[
{^\Fp}\tau_{\leq k + \dim B}Rf_{*}\BQ_{M} \to Rf_{*}\BQ_{M}
\]
whose orthogonal complement is given by $[\Delta_{M/B}] - \Fp_{k}$. This reduces the multiplicativity of the perverse truncation functor to verifying the vanishing
\[
 \left([\Delta_{M/B}] - \Fp_{k+l}\right) \circ [\Delta^{\mathrm{sm}}_{M/B}] \circ (\Fp_k \times \Fp_l)= 0.
\]
As explained in \cite[Section 2.5.3]{MSY}, this vanishing follows from Axiom D, namely from the assumption that the support of the convolution kernel has codimension $g$. In particular, for parts (a) and (b), we do not require the full strength of Axiom D, but only the codimension condition.

The last condition of Axiom D, that the purely codimension $g$ support of the convolution kernel is $\overline{\Gamma}_{\mathrm{mult}}$, was only used in verifying (c); we refer to \cite[Proof of Theorem 1.5]{BMSY2} for more details.
\end{proof}

\subsection{Applications}
We discuss several applications of the theory of dualizable abelian fibrations.
\medskip

\noindent {\bf Motivic decompositions.} The results above allow us to prove Conjecture \ref{conj3.1} for many abelian fibrations, including all the examples given in Section \ref{sec2.3}. Furthermore, we obtain the following result which at a first glance is not a direct consequence of Theorem \ref{thm3.2}.

\begin{thm}[\cite{MSY2}]\label{thm3.5}
Let 
\[
h: \mathrm{Higgs}_{n,d} \to \CA
\]
be the Hitchin system associated with a curve $\Sigma$ of genus $g\geq 2$, and coprime rank $n$ and degree~$d$. Then Conjecture \ref{conj3.1} holds for $h: \mathrm{Higgs}_{n,d} \to \CA$.
\end{thm}

As we mentioned in Section \ref{sec2.3}, the Hitchin system above is not a dualizable abelian fibration. The proof of Theorem \ref{thm3.5} is obtained by reducing it to the case of parabolic Hitchin system $f_d: M^\eta_{n,d} \to B$ using Springer theory and nearby/vanishing cycle techniques. The latter is dualizable by Theorem \ref{thm2.9} whose motivic decomposition is given by Theorem \ref{thm3.2}(a).

\medskip
\noindent {\bf Multiplicativity.}
Let $C_0$ be an integral locally planar curve, and let $\overline{J}_{C_0}$ be its compactified Jacobian; it is in general a very singular variety. Maulik--Yun \cite{MY} introduced a perverse filtration on the cohomology
\[
P_0H^*(\overline{J}_{C_0}, \BQ) \subset P_1H^*(\overline{J}_{C_0}, \BQ) \subset \cdots \subset H^*(\overline{J}_{C_0}, \BQ).
\]
This perverse filtration is conjectured to have deep connections to knot invariants associated with the singular points of $X$ \cite{ORS, Shen_survey}, Using techniques in representation theory, Oblomkov--Yun proved in \cite{OY} that the perverse filtration is multiplicative with respect to the cup-product for rational curves with a singular point given by 
\[
x^p-y^q, \quad \mathrm{gcd}(p,q)=1.
\]
The theory of dualizable abelian fibrations yields a proof of the multiplicativity for any integral locally planar curve $C_0$; see \cite[Theorem 0.7]{MSY}.

\begin{thm}[\cite{MSY}]\label{thm3.6}
    The perverse filtration $P_\bullet H^*(\overline{J}_{C_0}, \BQ)$ is multiplicative, \emph{i.e.},
    \[
\cup:    P_iH^*(\overline{J}_{C_0}, \BQ) \times P_{j}H^*(\overline{J}_{C_0}, \BQ) \to P_{i+j} H^*(\overline{J}_{C_0}, \BQ).
    \]
\end{thm}

Recall that the perverse filtration of Maulik--Yun is defined as follows. By embedding the curve $C_0$ into an appropriate family $C \to B$ as a fiber over a point $0\in B$, we obtain the Cartesian diagram
\[
\begin{tikzcd}
\overline{J}_{C_0} \arrow[r, hook, ""] \arrow[d, ""] & \overline{J}_C \arrow[d, " f"] \\
\{0\} \arrow[r, hook, " "] & B.
\end{tikzcd}
\] 
The perverse filtration for the fiber over $0\in B$ is then defined by the perverse truncation functor applied to $Rf_* \BQ_{\overline{J}_C}$. Therefore, Theorem \ref{thm3.6} can be deduced from Theorem \ref{thm2.6} and Theorem \ref{thm3.2}(b).

\medskip
\noindent {\bf Perverse vs.~Chern.} A key step in the proof of Theorem \ref{thm3.2} is the characterization (\ref{P_cycle}) of the perverse truncation functor using the Fourier transform. As a consequence, we obtain the following; see \cite[Theorem 0.1]{MSY}.

\begin{thm}[\cite{MSY}]\label{thm3.77}
    Let $f: M\to B$ be a dualizable abelian fibration with a dual fibration $f^\vee: M^\vee \to B$. Then we have
    \[
\mathrm{Im}\left(\mathfrak{F}_i: H^*(M^\vee, \BQ) \to H^*(M, \BQ)\right) \subset P_iH^*(M, \BQ).
    \]
\end{thm}

The reader may compare Theorem \ref{thm3.77} with (\ref{P=C_0}).

By a similar argument as in Section \ref{sec1.4}, Theorem \ref{thm3.77} combined with Theorem \ref{thm3.3}(a) leads to $P=C$ type results:
\begin{enumerate}
    \item[(a)] Applying Theorem \ref{thm3.77} to the parabolic Hitchin system $f_d: M^\eta_{n,d} \to B$ yields a proof of the $P=W$ conjecture \cite{dCHM, MS_PW, HMMS}; see \cite[Proof of Theorem 0.4]{MSY}.
    \item[(b)] Applying Theorem \ref{thm3.77} to the Le Potier moduli space of $1$-dimensional semistable sheaves on a del Pezzo surface yields a proof of the $P=C$ conjecture \cite{KPS, KLMP} for low degree cohomology; see \cite[Theorem 0.6]{MSY} and \cite[Theorems 0.3 and 0.6]{PSSZ}.
\end{enumerate}

\begin{rmk}
    More accurately, in both (a, b) above, we apply a variant of Theorem \ref{thm3.7}; see the discussions in the paragraph after Theorem \ref{thm2.9} and Section \ref{sec3.4}.
\end{rmk}
\medskip
\noindent {\bf Intrinsic cohomology rings.} Theorem \ref{thm3.3}(b) can be applied in the following interesting way: if we have two abelian fibrations $f_1: M_1 \to B$ and $f_2: M_2 \to B$ which share a common dual $f^\vee: M^\vee \to B$, then we must have an isomorphism of the bigraded $H^*(B, \BQ)$-algebras:
\[
\left(\bigoplus_{i,d} \mathrm{Gr}^P_iH^d(M_1, \BQ), \overline{\cup} \right) \simeq \left(\bigoplus_{i,d} \mathrm{Gr}^P_iH^d(M_2, \BQ), \overline{\cup} \right);
\]
this is because both sides are isomorphic to the convolution algebra on the associated graded cohomology of $M^\vee$ given by the class $[\overline{\Gamma}_{\mathrm{mult}}]$. This leads to the following theorem proven in~\cite[Theorem 0.6]{BMSY2}, which constructs the \emph{intrinsic cohomology ring} of the universal compactified Jacobians.

\begin{thm}[\cite{BMSY2}]\label{thm3.7}
   Assume $n\geq 1$. Let $f^\phi: \overline{J}^\phi_{g,n} \to \overline{\CM}_{g,n}$ be the universal fine compactified Jacobian associated with a nondegenerate stability condition $\phi$. Then the induced associated graded cohomology ring
   \[
\left(\bigoplus_{i,d} \mathrm{Gr}^P_iH^d(\overline{J}^\phi_{g,n}, \BQ),~~~ \overline{\cup} \right),
\]
as a bigraded $H^*(B, \BQ)$-algebra, does not depend on the nondegenerate stability condition $\phi$.
\end{thm}

As discussed above, Theorem \ref{thm3.7} is an immediate consequence of Theorem \ref{thm3.3}(b) and Theorem \ref{thm2.8}. On the other hand, it was proven in \cite[Theorem 0.5]{BMSY2} that the cohomology ring
\[
\left(H^*(\overline{J}^\phi_{g,1}, \BQ), \cup \right)
\]
is sensitive to $\phi$ as long as $g\geq 4$; these are exactly the examples in Theorem \ref{thm2.2}.

The results above suggest that the associated graded cohomology ring behaves ``better'' than the usual cohomology ring for dualizable abelian fibrations.

\subsection{Variants}\label{sec3.4}

To apply the theory of dualizable abelian fibrations, some variants are considered.

\medskip 

\noindent {\bf Torsors.} A very common situation is that the nonsingular fibers of an abelian fibration do not form an abelian scheme, but rather a torsor under an abelian scheme. Such cases were studied in \cite[Section 4]{MSY}, \cite{MSY2}, and \cite[Section 1]{BMSY2}, where the framework of dualizable abelian fibrations continues to apply with only minor modifications.

\medskip

\noindent {\bf Smaller support.} The full support assumption in Axiom B is only needed to verify that
\begin{equation*}
\textup{the correspondence \eqref{P_cycle} realizes the perverse truncation functor}.    \tag{$\ast$}
\end{equation*}
Therefore, if there are other methods to verify $(\ast)$, Axiom B can be removed. For example, if the structure of the support is not so complicated, one may try to verify ($\ast$) for the general point of \emph{every} support. Alternatively, nearby/vanishing cycle techniques were applied in \cite[Section 4]{MSY2} to verify ($\ast$) for certain Lagrangian fibrations.

\medskip 

\noindent {\bf Enhanced Axiom C+,}
In all the examples in Section \ref{sec2.3}, the following enhanced version of Axiom C+ is satisfied:

\noindent Assume that $f: M\to B$ and $f^\vee: M^\vee \to B$ given by Axiom A are $\delta$-regular weak abelian fibrations in the sense of Ng\^o \cite{Ngo} with the commutative group schemes $P\to B$ and $P^\vee \to B$. For any $P$-equivariant object $\CF \in D^b\mathrm{Coh}(M)$ and $P^\vee$-equivariant object $\CG \in D^b\mathrm{Coh}(M^\vee)$, the objects
\[
\CP^{-1} \circ \left(p^*\CF \otimes \CP^{\otimes N}  \otimes q^*\CG\right), \quad  \left(p^*\CF \otimes \CP^{\otimes N} \otimes q^*\CG\right)\circ \CP^{-1}
\]
have support of codimension $\geq g$ for infinitely many integers $N \geq 1$. Here $p,q$ are the natural projections from $M^\vee \times_B M$ to $M^\vee$ and $M$ respectively.

The proof of \cite[Theorem 0.1]{Ghosh} gives the following theorem.

\begin{thm}[Ghosh \cite{Ghosh}]\label{thm3.10}
    Assume that $f: M \to B$ is a dualizable abelian fibration satisfying the enhanced Axiom C+. Then the Chern classes of the tangent bundle $T_M$ satisfy
    \[
    c_i(T_M) \in P_i H^{2i}(M, \BQ).
    \]
\end{thm}

We refer to \cite{Ghosh} for more details, which also proved some motivic versions of Theorem \ref{thm3.10}.
 
\subsection{Lagrangian fibrations}
We conclude the discussion with Lagrangian fibrations. They form a distinguished class among all abelian fibrations and, from the perspective of the decomposition theorem, exhibit behavior that is closest to that of abelian schemes.

We say that $f: M \to B$ is a Lagrangian fibration if 
\begin{enumerate}
    \item[(a)] both $M, B$ are nonsingular, 
    \item[(b)] $f$ is proper and flat, and
    \item[(b)] a general fiber $F$ is an abelian variety of half dimension, such that there is a holomorphic symplectic form $\sigma$ on $M$ whose restriction to $F$ vanishes.
\end{enumerate}

We note that if $M$ is a projective irreducible symplectic manifold and $0< \dim B < \dim M$, then $f$ is automatically a Lagrangian fibration \cite{Mat}. For our definition of Lagrangian fibrations, we do \emph{not} assume that $M$ is projective or proper.

\begin{conj}[\cite{BMSY1, Shen_ICM}]\label{conj}
Let $f: M \to B$ be a Lagrangian fibration. The following hold.
\begin{enumerate}
    \item[(a)]  The morphism $f: M \to B$ admits a multiplicative decomposition theorem
    \[
    Rf_* \BQ_M \simeq \bigoplus_{i} \CH_{(i)},\quad \CH_{(i)} := {^\Fp}H^{i + \dim B}(Rf_{*}\BQ_M)[-i - \dim B] \in D^b_c(B)    \]
    \item[(b)] Let 
    \[
    H^*(M, \BQ) = \bigoplus_{i} H_{(i)}^*(M, \BQ)
    \]
    be the cohomological decomposition given by an isomorphism in (a). Then 
    \[
    c_i(T_M) \in H^{2i}_{(i)}(M, \BQ).
    \]
\end{enumerate}
\end{conj}

If Conjecture \ref{conj3.1} holds, then the perverse filtration for $f: M \to B$ admits a multiplicative splitting of the perverse filtration as in (b). This was known when $M$ is a projective irreducible symplectic manifold by \cite[Appendix A]{SY0}, or when $f: M \to B$ is the Hitchin system, where it follows from the $P=W$ conjecture; see \cite[Sections 0.3]{dCMS}. 

The reader may compare Theorem \ref{thm3.10} and Conjecture \ref{conj}(b).

\begin{example}
    For an abelian scheme $f: A\to B$, the Chern classes of $T_A$ are pulled back from the base $B$. Therefore we clearly have
    \[
    c_{i}(T_A) \in P_0H^{2i}(A, \BQ).
    \]
    Conjecture \ref{conj}(b) predicts that, if $f: A\to B$ is Lagrangian, then
    \[
    c_i(T_A)= 0, \quad i >0.
    \]
    This indeed holds by the Mumford relations for the Hodge bundle \cite[Corollary 5.3]{Mumford}.
\end{example}

Finally, we note that Conjecture \ref{conj} is expected to admit a motivic enhancement: there exists a multiplicative motivic decomposition
\[
h(M)  = \bigoplus_{i} h_i(M), \quad \cup: h_i(M) \times h_{j}(M) \to h_{i+j}(M)
\]
in the category of relative Chow motives over $B$, and the Chern classes satisfy
\[
c_i(T_M) \in \mathrm{CH}^{i}( h_{i}(M) ).
\]
We refer to \cite[Section 3]{BMSY1} for further discussions on connections to conjectures of Beauville and Voisin \cite{B2, Voi0, Voi1}.

\end{document}